\documentclass[11pt]{nyj}

\title{$K$-theory of weight varieties} 

\author{Ho-Hon Leung}



\newcommand{\func}[3]{#1 \colon #2 \to #3}

\theoremstyle{plain}
\newtheorem{theorem}{Theorem}[section]

\newtheorem{corollary}[theorem]{Corollary}
\newtheorem{lemma}[theorem]{Lemma}

\theoremstyle{definition}

\newtheorem{definition}[theorem]{Definition}

\newtheorem*{examplen}{Example}

\begin{document} 
 
\begin{abstract}  

Let $T$ be a compact torus and $(M,\omega)$ a Hamiltonian $T$-space. We give a new proof of the $K$-theoretic analogue of the Kirwan surjectivity theorem in symplectic geometry (see \cite{HaradaLandweber}) by using the equivariant version of the Kirwan map introduced in \cite{Goldin}. We compute the kernel of this equivariant Kirwan map, and hence give a computation of the kernel of the Kirwan map. As an application, we find the presentation of the kernel of the Kirwan map for the $T$-equivariant $K$-theory of flag varieties $G/T$ where $G$ is a compact, connected and simply-connected Lie group. In the last section, we find explicit formulae for the $K$-theory of weight varieties.

\end{abstract} 
\maketitle
\tableofcontents

\section{Introduction} \label{section1}

For $M$ a compact Hamiltonian $T$-space, where $T$ is a compact torus, we have a moment map $\phi\colon M\rightarrow\mathfrak{t}^{\ast}$. For any regular value $\mu$ of $\phi$, $\phi^{-1}(\mu)$ is a submanifold of $M$ and has a locally free $T$-action by the invariance of $\phi$. The \emph{symplectic reduction} of $M$ at $\mu$ is defined as $M//T(\mu):=\phi^{-1}(\mu)/T$. The parameter $\mu$ is surpressed when $\mu=0$. Kirwan \cite{Kirwan} proved that the natural map, which is now called the \emph{Kirwan map}, \[\kappa\colon H_T^\ast(M;\mathbb{Q})\rightarrow H_T^\ast(\phi^{-1}(0);\mathbb{Q})\cong H^\ast(M//T;\mathbb{Q})\] induced from the inclusion $\phi^{-1}(0)\subset M$ is a surjection when $0\in\mathfrak{t}^\ast$ is a regular value of $\phi$. This result was done in the context of rational Borel equivariant cohomology. In the context of complex $K$-theory, a theorem of Harada and Landweber \cite{HaradaLandweber} showed that \[\kappa\colon K_{T}^{\ast}(M)\rightarrow K_{T}^{\ast}(\phi^{-1}(0))\] is a surjection. This result was done over $\mathbb{Z}$. 

In Section \ref{section2}, we give another proof of this theorem by using \emph{equivariant Kirwan map}, which was first introduced by Goldin \cite{Goldin} in the context of rational cohomology. It can also be seen as an equivariant version of the Kirwan map.

\begin{theorem} \label{theorem1.1}
Let $T$ be a compact torus and $M$ be a compact Hamiltonian $T$-space with moment map $\phi\colon M\rightarrow\mathfrak{t}^{\ast}$. Let $S$ be a circle in $T$, and $\phi|_S:=M\rightarrow\mathbb{R}$ be the corresponding component of the moment map. For a regular value $0\in\mathfrak{t}^\ast$ of $\phi|_S$, the equivariant Kirwan map \[\kappa_{S}\colon K_T^\ast(M)\rightarrow K_T^\ast(\phi|_{S}^{-1}(0))\]is a surjection.
\end{theorem}

As an immediate corollary of a result in \cite{HaradaLandweber}, we also find the kernel of this equivariant Kirwan map.

In Section \ref{section5}, for the special case $G=SU(n)$, we find an explicit formula for the $K$-theory of weight varieties, the symplectic reduction of flag varieties $SU(n)/T$. The main result is Theorem \ref{theorem5.3}. The results in this section are the $K$-theoretic analogues of \cite{GGoldin}.

\section{Equivariant Kirwan map in $K$-theory} \label{section2}

First we recall the basic settings of the subject. Let $G$ be a compact connected Lie group. A \emph{compact Hamiltonian $G$-space} is a compact symplectic manifold $(M,\omega)$ on which $G$ acts by symplectomorphisms, together with a $G$-equivariant moment map $\phi\colon M\rightarrow\mathfrak{g}^\ast$ satisfying Hamilton's equation:\[\langle\mathrm{d}\phi,X\rangle=\iota_{X'}\omega,\forall X \in \mathfrak{g}\] where $G$ acts on $\mathfrak{g}^\ast$ by the coadjoint action and $X'$ denotes the vector field on $M$ generated by $X \in \mathfrak{g}$. In this paper, we only deal with a compact torus action, so we will use the $T$-action on $M$ as our notation instead. Let $T'$ be a subtorus in $T$, $\phi|_{T'}\colon M\rightarrow\mathfrak{t'}^\ast$ is the restriction of the $T$-action to the $T'$-action. We call $\phi|_{T'}$ the component of the moment map corresponding to $T'$ in $T$.

We fix the notations about Morse theory. Let $f\colon M\rightarrow\mathbb{R}$ be a Morse function on a compact Riemannian manifold $M$. Consider its negative gradient flow on $M$, let $\{C_i\}$ be the connected components of the critical sets of $f$. Define the stratum $S_i$ to be the set of points of $M$ which flow down to $C_i$ by their paths of steepest descent. There is an ordering on $I$: $i\leq j$ if $f(C_i)\leq f(C_j)$. Hence we obtain a smooth stratification of $M=\cup S_i$. For all $i,j \in I$, denote \[M_i^{+}=\bigcup_{j\leq i}S_j,\quad M_i^{-}=\bigcup_{j<i}S_j\]As we are working in the equivariant category, we require that the Morse function and the Riemannian metric to be $T$-invariant.

In the following, we will consider the norm square of the moment map. In general, it is not a Morse function due to the possible presence of singularities of the critical sets. But the norm square of the moment map still yields a smooth stratifications and the results of Morse-Bott theory still holds in this general setting (Such functions are now called the Morse-Kirwan functions). For the descriptions and properties of these functions, see \cite{Kirwan}. Kirwan proved that the Morse-Kirwan functions are equivariantly perfect in the context of rational cohomology. For more results in this direction, see \cite{Kirwan} and \cite{Lerman}. In the context of equivariant $K$-theory, the following result is shown in \cite{HaradaLandweber}:

\begin{lemma}[Harada and Landweber] \label{lemma2.1}
Let $T$ be a compact torus and $(M,\omega)$ be a compact Hamiltonian $T$-space with moment map $\phi\colon M\rightarrow\mathfrak{t}^\ast$. Let $f=||\phi||^2$ be the norm square of the moment map. Let $\{C_i\}$ be the connected components of the critical sets of $f$ and the stratum $S_i$ be the set of points of $M$ which flow down to $C_i$ by their paths of steepest descent. The inclusion $C_i\to S_i$ of a critical set into its stratum induces an isomorphism $K_T^\ast(S_i)\cong K_T^\ast(C_i)$.
\end{lemma}

For a smooth stratification $M=\cup S_i$ defined by a Morse-Kirwan function $f$, i.e. the strata $S_i$ are locally closed submanifolds of $M$ and each of them satisfies the closure property $\overline{S}_i\subseteq M_i^{+}$. We have a $T$-normal bundle $N_i$ to $S_i$ in $M$. By excision, we have \[K_T^\ast(M_i^{+},M_i^{-})\cong K_T^\ast(N_i,N_i\backslash S_i)\]If $N_i$ is complex, by Thom Isomorphism we have \[K_T^\ast(N_i,N_i\backslash S_i)\cong K_T^{\ast-d(i)}(S_i)\]where the degree $d(i)$ of the stratum is the rank of its normal bundle $N_i$. Since the collection of the sets $M_i^{+}$ gives a filtration of $M$, we obtain a filtration of $K_T^\ast(M)$ and a spectral sequence \[E_1=\bigoplus_{i\in I}K_T^\ast(M_i^{+},M_i^{-})=\bigoplus_{i\in I}K_T^{-d(i)}(S_i),\quad E_\infty=\mbox{Gr}K_T^\ast(M)\]which converges to the associated graded algebra of the equivariant $K$-theory of $M$. By Lemma \ref{lemma2.1}, the spectral sequence becomes \[E_1=\bigoplus_{i \in I}K_T^{\ast-d(i)}(C_i),\quad E_\infty=\mbox{Gr}K_T^\ast(M)\]

\begin{definition}
The function $f$ is called \emph{equivariantly perfect} for equivariant $K$-theory if the above spectral sequence for equivariant $K$-theory collapses at the $E_1$ page, or equivalently speaking, we have the following short exact sequences: \[0\longrightarrow K_T^{\ast-d(i)}(C_i)\longrightarrow K_T^\ast(M_i^{+})\longrightarrow K_T^\ast(M_i^-)\longrightarrow 0\]for each $i \in I$.
\end{definition}

In \cite{HaradaLandweber}, Harada and Landweber showed the following theorem. (Indeed, they showed it for a compact Lie group $G$. But in our paper, we only need to consider the abelian case.)

\begin{theorem}[Harada and Landweber] \label{theorem2.3}
Let $T$ be a compact torus and $(M,\omega)$ be a compact Hamiltonian $T$-space with moment map $\phi\colon M\rightarrow\mathfrak{t}^\ast$. The norm square of the moment map $f=||\phi||^2$ is an equivariantly perfect Morse-Kirwan function for equivariant $K$-theory. By Bott periodicity in complex equivariant $K$-theory, we can rewrite the short exact sequences as: \[0\longrightarrow K_T^{\ast}(C_i)\longrightarrow K_T^\ast(M_i^{+})\longrightarrow K_T^\ast(M_i^-)\longrightarrow 0\]
\end{theorem}

Let $\phi|_S\colon M\rightarrow\mathbb{R}$ be the component of the moment map $\phi$ corresponding to a circle $S$ in $T$. Equivalently we are considering a compact Hamiltonian $S$-manifold with the moment map $\phi|_S$. By Theorem \ref{theorem2.3} above, the norm square of the moment map $||\phi|_S||^2$ is an equivariantly perfect Morse-Kirwan function for equivariant $K$-theory. We can now give our proof of Theorem \ref{theorem1.1}.

\begin{proof}[Proof of Theorem \ref{theorem1.1}]
Our proof is essentially the $K$-theoretic analogue of Theorem 1.2 in \cite{Goldin}. For the Morse-Kirwan function $f=||\phi|_{S}||^2$, denote $C_0 =f^{-1}(0)=\phi|_S^{-1}(0)$.

First, we need to show that $K_T^\ast(M_i^+)\rightarrow K_T^\ast(\phi|_S^{-1}(0))$ is surjective for all $i \in I$. We will show it by induction. 
 
Notice that $K_T^\ast(M_0^{+})\cong K_T^\ast(C_0)=K_T^\ast(\phi|_S^{-1}(0))$ by Theorem \ref{theorem2.3}. Assume the inductive hypothesis that $K_T^\ast(M_i^{+})\rightarrow K_T^\ast(C_0)$ is surjective for $0\leq i \leq k-1$. By the equivariant homotopy equivalence, we have \[K_T^\ast(M_k^-)\cong K_T^\ast(M_{k-1}^+)\]Hence, we now have the surjection of \begin{equation} \label{equation1} K_T^\ast(M_k^{-})\cong K_T^\ast(M_{k-1}^{+})\rightarrow K_T^\ast(C_0) \end{equation}
By Theorem \ref{theorem2.3}, we know that $K_T^\ast(M_i^{+})\rightarrow K_T^{\ast}(M_i^{-})$ is a surjection for each $i$. By equation (\ref{equation1}), $K_T^\ast(M_k^{+})\rightarrow K_T^\ast(C_0)$ is a surjection and hence our induction works.

Given that $K_T^\ast(M)=K_T^\ast(\varinjlim M_i^{+})=\varprojlim K_T^\ast(M_i^{+})$, these equalities hold because we have the surjections $K_T^\ast(M_i^{+})\rightarrow K_T^\ast(M_i^{-})$ for all $i$. Hence we have the surjection result for $\kappa_S \colon K_T^\ast(M) \rightarrow K_T^\ast(C_0)=K_T^\ast(\phi|_S^{-1}(0))$, as desired.
\end{proof}

\begin{corollary} \label{corollary2.4}
Let $T$ be a compact torus and $M$ be a compact Hamiltonian $T$-space with moment map $\phi\colon M\rightarrow\mathfrak{t}^\ast$. Suppose that $T$ acts freely on the zero level set of the moment map. Then \[\kappa\colon K_T^\ast(M)\rightarrow K^\ast(M//T)\]is a surjection.
\end{corollary}

\begin{proof} Choose a splitting of $T=S_1\times S_2\times...\times S_{\mbox{dim}T}$ where each $S_i$ is quotiented out one at a time. Since $T$ acts freely on the zero level set of the moment map, by Theorem \ref{theorem1.1}, we have \[\kappa_{S_1}\colon K_T^\ast(M)\rightarrow K_T^\ast(\phi|_{S_1}^{-1}(0))\cong K_{T/S_1}^\ast(M//S_1)\]is a surjection. By reduction in stages, we have \[K_T^\ast(M)\rightarrow K_{T/S_1}^\ast(M//S_1)\rightarrow K_{T/(S_1\times S_2)}^\ast(M//(S_1\times S_2))\rightarrow ... \rightarrow K_{T/T}^\ast(M//T)=K^\ast(M//T)\]as desired.
\end{proof}

We compute the kernel of our equivariant Kirwan map, which can be seen as a $K$-theoretic analogue of \cite{Goldin}.

\begin{theorem} \label{corollary3.1}
Let $T$ be a compact torus and $M$ be a compact Hamiltonian $T$-space with moment map $\phi\colon M\rightarrow\mathfrak{t}^\ast$. Let $T'$ be a subtorus in $T$. Let $\phi|_{T'}$ be the corresponding moment map for the Hamiltonian $T'$-action on $M$. For 0 a regular value of $\phi|_{T'}$, the kernel of the equivariant Kirwan map \[\kappa_{T'}\colon K_T^\ast(M)\rightarrow K_T^\ast(\phi|_{T'}^{-1}(0))\]is the ideal $\langle K_{T}^{\mathfrak{t}'}\rangle$ generated by $K_{T}^{\mathfrak{t}'}=\cup_{\xi\in \mathfrak{t}'}K_{T}^{\xi}$ where \[K_T^\xi=\{\alpha\in K_T^\ast(M) \mid \alpha|_C=0 \mbox{ for all connected components } C \mbox{ of } M^{T} \mbox{ with } \langle\phi(C),\xi\rangle\leq 0\}\]
\end{theorem}

\begin{proof} Choose a splitting of $T'=S_1\times S_2\times...\times S_{\mbox{dim}T'}$ where each $S_i$ is quotiented out one at a time. By Theorem 3.1 in \cite{HHaradaLandweber}, the kernel of the equivariant Kirwan map $\kappa_{S_i}$ is generated by $K_T^\xi$ and $K_T^{-\xi}$ for a choice of generator $\xi\in\mathfrak{s}_i$. By successive application of this result to each $S_i$ where $i=1,2,3,\cdots\mbox{dim}T'$, we get our desired result.
\end{proof}

\section{$K$-theory of weight variety} \label{section5}
\subsection{Weight varieties} \label{section5.1}
If $G=SU(n)$, we can naturally identify the set of Hermitian matrices $H$ with $\mathfrak{g}^\ast$ by the trace map, i.e. $tr\colon (H)\rightarrow\mathfrak{g}^\ast$ defined by $A\mapsto i.tr(A)$. So $\lambda\in\mathfrak{t}^\ast$ is a real diagonal matrix with entries $\lambda_1,\lambda_2,...,\lambda_n$ in the diagonal. Through this identification, $M=\mathcal{O}_{\lambda}$ is an adjoint orbit of $G$ through $\lambda$. The moment map corresponding to the $T$-action on $\mathcal{O}_{\lambda}$ takes a matrix to its diagonal entries, call it $\mu\in\mathfrak{t}^\ast$. Hence, $\mathcal{O}_{\lambda}//T(\mu)$, $\mu\in\mathfrak{t}^{\ast}$ is the symplectic quotient by the action of diagonal matrices at $\mu\in\mathfrak{t}^{\ast}$. The symplectic quotient consists of all Hermitian matrices with spectrum $\lambda=(\lambda_1,\lambda_2,...,\lambda_n)$ and diagonal entries $\mu=(\mu_1,\mu_2,...,\mu_n)$. We call this symplectic quotient $\mathcal{O}_{\lambda}//T(\mu)$ a \emph{weight variety}.

If $\lambda=(\lambda_1,\lambda_2,...,\lambda_n)$ has the property that all entries have distinct values, then $\mathcal{O}_{\lambda}$ is a generic coadjoint orbit of $SU(n)$. It is symplectomorphic to a complete flag variety in $\mathbb{C}^n$. In this section, we mainly deal with the generic case unless otherwise stated. For more about the properties of weight varieties, see \cite{Knutson}. For the Weyl element action of any $\gamma\in W$ on $\lambda\in\mathfrak{t}^{\ast}$, we are going to use the notation $\lambda_{\gamma}=(\lambda_{\gamma^{-1}(1)},...,\lambda_{\gamma^{-1}(n)})$ in our proofs for our notational convenience.

\subsection{Divided difference operators and double Grothendieck polynomials} \label{section5.2}

Let $f$ be a polynomial in $n$ variables, call them $x_1,x_2,...,x_n$ (and possibly some other variables), the \emph{divided difference operator} $\partial_i$ is defined as \[\partial_i f(...,x_i,x_{i+1},...)=\frac{f(...,x_i,x_{i+1},...)-f(...x_{i+1},x_i,...)}{x_i -x_{i+1}}\]The \emph{isobaric divided difference operator} is \[\pi_i(f)=\partial_i(x_i f)=\frac{x_i f(...,x_i,x_{i+1},...)-x_{i+1}f(...,x_{i+1},x_i,...)}{x_i -x_{i+1}}\]The \emph{top Grothendieck polynomial} is \[G_{id}(x,y)=\prod_{i<j}(1-\frac{y_j}{x_i})\]Note that the isobaric divided difference operator acts on $G_{id}$ naturally by $\pi_i(G_{id})$. And $\pi_i(P.Q)=\pi_i(P)Q$ provided that $Q$ is a symmetric polynomial in $x_1,x_2,...x_n$. So this operator preserves the ideal generated by all differences of elementary symmetric polynomials $e_i(x_1,...,x_n)-e_i(y_1,...,y_n)$ for all $i=1,...,n$, denote this ideal by $I$. That is, the operator $\pi_i$ acts on the ring $R$ defined by \[R=\frac{\mathbb{Z}[x_1^{\pm 1},...,x_n^{\pm 1},y_1^{\pm 1},...,y_n^{\pm 1}]}{I}\]

For any element $\omega\in S_n$, $\omega$ has reduced word expression $\omega=s_{i_1}s_{i_2}...s_{i_l}$ (where each $s_{i_j}$ is a transposition between $i_j,i_{j+1}$). We can define the corresponding operator:\[\pi_{s_{i_1}s_{i_2}...s_{i_l}}=\pi_{s_{i_1}}...\pi_{s_{i_l}}\]which is independent of the choice of the reduced word expression.

For any $\mu\in S_n$, the \emph{double Grothendieck polynomial} $G_{\mu}$ is: \[\pi_{\mu^{-1}}G_{id}=G_\mu\]Define the permuted double Grothendieck polynomials $G_{\omega}^{\gamma}$ by \[G_{\omega}^{\gamma}(x,y)=G_{\gamma^{-1}\omega}(x,y_{\gamma})=\pi_{\omega^{-1} \gamma}G_{id}(x,y_{\gamma})\]where $y_{\gamma}$ means the permutation of the $y_1,...,y_n$ variables by $\gamma$.

\begin{examplen}  For $G=SU(3),W=S_3$, we have \[G_{id}=(1-\frac{y_2}{x_1})(1-\frac{y_3}{x_1})(1-\frac{y_3}{x_2})\]
\begin{eqnarray}
G_{(23)}^{(12)}&=&\pi_{(23)(12)}G_{id}(x,y_{(12)})\nonumber\\
&=&\pi_{(23)(12)}\left(1-\frac{y_3}{x_1}\right)\left(1-\frac{y_1}{x_1}\right)\left(1-\frac{y_3}{x_2}\right)\nonumber\\
&=&\pi_{(23)}\left(\frac{x_1\left(1-\frac{y_3}{x_1}\right)\left(1-\frac{y_1}{x_1}\right)\left(1-\frac{y_3}{x_2}\right)-x_2\left(1-\frac{y_3}{x_2}\right)\left(1-\frac{y_1}{x_2}\right)\left(1-\frac{y_3}{x_2}\right)}{x_1-x_2}\right)\nonumber\\
&=&\pi_{(23)}\left(1-\frac{y_3}{x_1}\right)\left(1-\frac{y_3}{x_2}\right)\nonumber\\
&=&\left(1-\frac{y_3}{x_1}\right)\nonumber
\end{eqnarray}
\end{examplen}

\subsection{$T$-equivariant $K$-theory of flag varieties} \label{section5.3}

We have the following formula for $K_T^\ast(SU(n)/T)$ (see \cite{wfulton}): \[
K_T^{\ast}(SU(n)/T)\cong R(T)\otimes_{R(G)}R(T)\cong R(T)\otimes_\mathbb{Z}R(T)/J\] where $R(G)\cong R(T)^W$ and $R(T)$ are the character rings of $G, T$, where $G=SU(n)$, respectively. $J\subset R(T)\otimes_{\mathbb{Z}}R(T)$ is the ideal generated by $a\otimes 1-1\otimes a$ for all elements $a\in R(T)^W$. $R(T)^W$ is the Weyl group invariant of $R(T)$.

$R(T)$ can be written as a polynomial ring: \[R(T)=K_T^{\ast}(pt)\cong\mathbb{Z}[a_1^{\pm 1},...,a_{n-1}^{\pm 1}]\]In the equation $K_T^{\ast}(X)=R(T)\otimes_{\mathbb{Z}}R(T)/J$, denotes the first copy of $R(T)$ by $\mathbb{Z}[y_1^{\pm 1},...,y_{n-1}^{\pm 1}]$ and the second copy of $R(T)$ by $\mathbb{Z}[x_1^{\pm 1},...,x_{n-1}^{\pm 1}]$. Then the ideal $J$ is generated by $e_i(y_1,...,y_{n-1})-e_i(x_1,...,x_{n-1}), i=1,...,n-1$, where $e_i$ is the $i$-th symmetric polynomial in the corresponding variables. Equivalently, \begin{equation} \label{equation2} 
K_T^{\ast}(Fl(\mathbb{C}^n))\cong\frac{\mathbb{Z}[y_1^{\pm 1},...,y_n^{\pm 1},x_1,...,x_n]}{(J,(\prod\nolimits_{i=1}^n y_i)-1)}\end{equation} Notice that $x_i^{-1},i=1,...,n$ can be generated by some elements in the ideal $J$, where $J$ is the ideal generated by $e_i(y_1,...,y_{n})-e_i(x_1,...,x_{n})$, for all $i=1,...,n$.  

Let $G^{\mathbb{C}}$ be the complexification of a compact Lie group $G$ and $B\subset G^{\mathbb{C}}$ be a Borel subgroup. In our case, $G=SU(n),G^{\mathbb{C}}=SL(n,\mathbb{C})$. Then $G/T\approx G^{\mathbb{C}}/B$. $G^{\mathbb{C}}/B$ consists of even-real-dimensional Schubert cells, $C_{\omega}$ indexed by the elements in the Weyl Group $W$. That is, \[C_{\omega}=B\omega B/B,\omega\in W\]The closures of these cells are called \emph{Schubert varieties}: \[X_{\omega}=\overline{B\omega B}/B,\omega\in W\]For each Schubert variety $X_{\omega} ,\omega\in W$, denotes the $T$-equivariant structure sheaf on $X_{\omega}\subset G^{\mathbb{C}}/B$ by $[\mathcal{O}_{X_\omega}]$. It extends to the whole of $G^{\mathbb{C}}/B$ by defining it to be zero in the complement of $X_\omega$. Since $[\mathcal{O}_{X_\omega}]$ is a $T$-equivariant coherent sheaf on $G^{\mathbb{C}}/B$, it determines a class in $K_0(T,G^{\mathbb{C}}/B)$, the Grothendieck group constructed from the semigroup whose elements are the isomorphism classes of $T$-equivariant locally free sheaves. The elements ${[\mathcal{O}_{X_\omega}]}_{\omega\in W}$ form a $R(T)$-basis for the $R(T)$-module $K_0(T,G^{\mathbb{C}}/B)$. Since there is a canonical isomorphism between $K_0(T,G^{\mathbb{C}}/B)$ and $K_T(G^{\mathbb{C}}/B)=K_T(G/T)$ (see \cite{KostantKumar}), by abuse of notation we also denote ${[\mathcal{O}_{X_{\omega}}]}_{\omega\in W}$ as a linear basis in $K_T^{\ast}(G/T)$ over $R(T)$. 

On the other hand, the double Grothendieck polynomials $G_\omega,\omega\in W$, as Laurent polynomials in variables ${x_i,y_i},i=1,2,...,n$ form a basis of $K_{T\times B}(pt)\cong R(T)\otimes_\mathbb{Z} R(T)$ over $K_T(pt)\cong R(T)$. By the equivariant homotopy principle, \[K_{T\times B}(pt)=K_{T\times B}(M_{n\times n})\]where $M_{n\times n}$ denote the set of all $n\times n$ matrices over $\mathbb{C}$. By a theorem of \cite{KnutsonMiller}, we are able to identify the classes generated by matrix Schubert varieties in $K_{T\times B}(M_{n\times n})$ (matrix Schubert varieties form a cell decomposition of $M_{n\times n}/B$) with the double Grothendieck polynomials in $K_{T\times B}(pt)$. The open embedding $\iota\colon GL(n,\mathbb{C})\to M_{n\times n}$ induces a map in equivariant $K$-theory: \[\iota^\ast\colon K_{T\times B}(M_{n\times n})\to K_{T\times B}(GL(n,\mathbb{C}))=K_T(GL(n,\mathbb{C})/B)=K_T(SU(n)/T)\]Under this map, the classes generated by the matrix Schubert varieties in $K_{T \times B}(M_{n \times n})$ are mapped to the classes, $[\mathcal{O}_{X_\omega}]\in K_T(SU(n)/T)$, of the corresponding Schubert varieties in $SU(n)/T$. By identifications of the double Grothendieck polynomials in $K_{T\times B}(pt)$ and the classes generated by the matrix Schubert varieties in $K_{T\times B}(M_{n\times n})$, the map $\iota^\ast$ sends the double Grothendieck polynomials to the $T$-equivariant structure sheaves ${[\mathcal{O}_{X_\omega}]}_{\omega\in W}$, as a $R(T)$-basis in $K_T(G/T)\cong R(T)\otimes_{R(G)}R(T)$. For more results about the geometry and combinatorics of double Grothendieck polynomials and matrix Schubert varieties, see \cite{KnutsonMiller}.

By abuse of notation, from now on, we will take the double Grothendieck polynomials  $G_\omega(x,y),\omega\in W$ as a basis in $K_T^\ast(SU(n)/T)$ over $R(T)$.  Under our notations, notice that the top double Grothendieck polynomial $G_{id}(x,y)$ corresponds to the $T$-equivariant structure sheaf $[\mathcal{O}_{X_{\omega_0}}]$, where $\omega_0\in W$ is the permutation with the longest length, i.e. $\omega_0=s_ns_{n-1}...s_3s_2s_1$. 

For more about $K$-theory and $T$-equivariant $K$-theory of flag varieties, for example, see \cite{wfulton} and \cite{KostantKumar}.

\subsection{Restriction of $T$-equivariant $K$-theory of flag varieties to the fixed-point sets} \label{section5.4}

Since flag variety is compact, $Fl(\mathbb{C}^n)^T$, the $T$-fixed set is finite. By \cite{HHaradaLandweber}, we have the Kirwan injectivity map, i.e. the map \[\iota^{\ast}\colon K_T^{\ast}(Fl(\mathbb{C}^n))\rightarrow K_T^{\ast}(Fl(\mathbb{C}^n)^T)\]induced by the inclusion $\iota$ from $Fl(\mathbb{C}^n)^T$ to $Fl(\mathbb{C})$ is injective. We compute the restriction explicitly here. Notice that $Fl(\mathbb{C}^n)^T$ is indexed by the elements in the Weyl group $W=S_n$. The $T$-action on $\mathbb{C}^n$ splits into a sum of 1-dimensional vector spaces, call them $l_1,...,l_n$. The fixed points of $T$-action are the flags which can be written as:\[p_{\omega}=\langle l_{\omega(1)}\rangle\subset\langle l_{\omega(1)},l_{\omega(2)}\rangle\subset\langle l_{\omega(1)},l_{\omega(2)},l_{\omega(3)}\rangle\subset...\subset\langle l_{\omega(1)},...,l_{\omega(n)}\rangle=\mathbb{C}^n\]where $\omega\in W$ and call \[p_{id}=\langle l_1\rangle\subset\langle l_1,l_2\rangle\subset\langle l_1,l_2,l_3\rangle\subset...\subset\langle l_1,...,l_n\rangle=\mathbb{C}^n\]the base flag of $\mathbb{C}^n$. The description of the restriction map is as follows:

\begin{theorem} \label{theorem5.1}
Let $p_{\omega}$ be a fixed point in $Fl(\mathbb{C}^n)^T$ as above. The inclusion $\iota_{\omega}\colon p_{\omega}\rightarrow Fl(\mathbb{C}^n)$ induces a restriction \[\iota_{\omega}^{\ast}\colon K_T^{\ast}(Fl(\mathbb{C}^n))\to K_T^{\ast}(p_{\omega})=R(T)=\mathbb{Z}[y_1^{\pm 1},...,y_n^{\pm 1}]\]such that $\iota_{\omega}^{\ast}\colon y_i^{\pm 1}\rightarrow y_i^{\pm 1}, \iota_{\omega}^{\ast}\colon x_i\rightarrow y_{\omega(i)},i=1,...,n$. Also, the inclusion map $\iota\colon Fl(\mathbb{C}^n)^T\rightarrow Fl(\mathbb{C}^n)$ induces a map \[\iota^{\ast}\colon K_T^{\ast}(Fl(\mathbb{C}^n))\to K_T^{\ast}(Fl(\mathbb{C}^n)^T)=\oplus_{p_{\omega},\omega\in W}\mathbb{Z}[y_1^{\pm 1},...,y_n^{\pm 1}]\]whose further restriction to each component in the direct sum is just the map $\iota_{\omega}^{\ast}$.
\end{theorem}

\begin{proof} Consider $K_T^{\ast}(Fl(\mathbb{C}^n))$ as a module over $K_T^{\ast}(pt)=\mathbb{Z}[y_1^{\pm 1},...,y_n^{\pm 1}]$, the map \[K_T^{\ast}(Fl(\mathbb{C}^n))\to K_T^{\ast}(p)\]induced by mapping any point $p$ into $Fl(\mathbb{C}^n)$ is a surjective $R(T)$-module homomorphism and $K_T^\ast (Fl(\mathbb{C}^n))$ has a linear basis over $K_T^\ast (p)=R(T)=\mathbb{Z}[y_1^{\pm 1},...,y_n^{\pm 1}]$. Hence we must have $\func{\iota_\omega^\ast}{y_i^{\pm 1}}{y_i^{\pm 1}},i=1,...,n$, for all $\omega\in W$.
To find the image of $x_i$ under $\iota_{\omega}^{\ast}$, first, notice that in $K_T^{\ast}(pt)$, $y_i=[pt\times\mathbb{C}_i]$. $\mathbb{C}_i$ corresponds to the action of  $T=S^1\times...\times S^1$ on the $i$-th copy of $\mathbb{C}^n=\mathbb{C}\times...\times\mathbb{C}$ with weight 1 and acting trivally on all the other copies of $\mathbb{C}$. More generally, $y_{\omega(i)}=[pt\times\mathbb{C}_{\omega(i)}]$. In $K_T^{\ast}(p_\omega)$, $y_{\omega(i)}=[p_\omega \times\mathbb{C}_{\omega(i)}]$, where $p_\omega \times\mathbb{C}_{\omega(i)}$ is the $T$-line bundle over the point $p_\omega$. By the Hodgkin's result (see \cite{Hodgkin}), $K_T^{\ast}(G/T)=R(T)\otimes_{R(G)}K_G^\ast(G/T)(\cong R(T)\otimes_{R(G)} R(T))$. Following our use of notations in \ref{section5.3}, $x_i$ comes from the second copy of $R(T)$ (which is isomorphic to $K_G^{\ast}(G/T)$ under our identification). Hence, each $x_i$ is the class represented by the $G$-line bundle $G\times_T\mathbb{C}_i$ over $G/T$. $T$ acts on $G\times\mathbb{C}_i$ diagonally and $G\times_T\mathbb{C}_i$ is the orbit space of the $T$-action. In particular, $x_i$ is a $T$-line bundle over $G/T$ by restriction of $G$-action to $T$-action. So, $\iota_\omega^\ast(x_i)$ is simply the pullback $T$-line bundle of the map $\iota_\omega\colon p_\omega\to Fl(\mathbb{C}^n)$. For $i=1$, we have $\iota_\omega^\ast(x_1)=[p_\omega \times\mathbb{C}_{\omega(1)}]= y_{\omega(1)}$. Similarly, $\iota_\omega^\ast(x_i)=y_{\omega(i)}$ for $i=2,...,n$. And hence the result.
\end{proof}

\subsection{Relations between double Grothendieck polynomials and the Bruhat Ordering} \label{section5.5}

Recall our definition of the permuted double Grothendieck polynomials $G_{\omega}^{\gamma}$ in Section \ref{section5.2}: \[G_{\omega}^{\gamma}(x,y)=G_{\gamma^{-1}\omega}(x,y_{\gamma})=\pi_{\omega^{-1} \gamma}G_{id}(x,y_{\gamma})\] where $y_{\gamma}$ indicates the permutation of $y_1,...,y_n$ by $\gamma$. For $\gamma\in W$, define the permuted Bruhat ordering by \[v\leq_\gamma \omega\Leftrightarrow \gamma^{-1}v\leq\gamma^{-1}\omega\]

Notice that the permuted Bruhat ordering is related to the Schubert varieties in the following way: Each of the $T$-fixed points of a Schubert variety $X_\omega$ sits in one Schubert cell $C_v$ (the interior of a Schubert variety) for $v\leq\omega$. So the $T$-fixed point set can be identified as: \[(X_\omega)^T=\{v\mid v\leq\omega\}\]For a fixed $\gamma\in W$, we can define the permuted Schubert varieties by \[X_\omega^\gamma=\overline{\gamma B\gamma^{-1}\omega B}/B\]for any $\omega\in W$. Then the $T$-fixed point set of $X_\omega^\gamma$ is \[(X_\omega^\gamma)^T=\{v\mid v\leq_\gamma \omega\}\]Notice that $\{X_\omega^\gamma\}_{\omega\in W}$ also forms a cell decomposition of $G^\mathbb{C}/B\approx G/T$.

We define the support of the permuted double Grothendieck polynomials by \[\mbox{Supp}(G_{\omega}^{\gamma})=\{z\in W\mid G_{\omega}^{\gamma}|_z \neq 0\}\]Here we consider $G_{\omega}^{\gamma}$ as an element in $K_T^{\ast}(Fl(\mathbb{C}^n))$ (see Section \ref{section5.3}). So $G_{\omega}^{\gamma}|_z$ is the image of $G_\omega^\gamma$ under the restriction of the Kirwan injective map at the point $z\in W$. That is, \[\iota^\ast|_z\colon K_T^\ast(Fl(\mathbb{C}^n))\to K_T^\ast(p_z)\]Notice that the restriction rule follows Theorem \ref{theorem5.1}. That is, \[G_{\omega}^{\gamma}(x,y)|_z =G_{\omega}^{\gamma}(x_1,x_2,...,x_n,y_1,...,y_n)|_z =G_{\omega}(y_{z(1)},y_{z(2)},...,y_{z(n)},y_1,...,y_n)\]

\begin{examplen} Using the same notations as in the example in \ref{section5.2}, $G_{(23)}^{(12)}=(1-\frac{y_3}{x_1})\in K_T^{\ast}(Fl(\mathbb{C}^3))$. There are six fixed points for each element in $S_3$, \[G_{(23)}^{(12)}|_{(23)}\neq 0, G_{(23)}^{(12)}|_{(123)}\neq 0, G_{(23)}^{(12)}|_{(13)}=0\] \[G_{(23)}^{(12)}|_{(132)}=0, G_{(23)}^{(12)}|_{(12)}\neq 0, G_{(23)}^{(12)}|_{id}\neq 0\]So the support of a permuted double Grothendieck polynomial contains $id,(12),(23),(123)$. On the other hand,
\begin{eqnarray}
(X_{(23)}^{(12)})^T &=& \{v\in S_3\mid (12)v\leq (12)(23)=(123)\}\nonumber\\
&=& \{v\in S_3\mid (12)v\leq id,(12),(23)\quad\mbox{or}\quad (123)\}\nonumber\\
&=& \{v\in S_3\mid v\leq (12),id,(123)\quad \mbox{or}\quad (23)\}\nonumber\\ \nonumber
\end{eqnarray} which is the same as $\mbox{Supp}(G_{(23)}^{(12)})$. 
\end{examplen}

Now we will show a fundamental relation between the permuted double Grothendieck polynomials and the permuted Bruhat Orderings: 

\begin{theorem} \label{theorem5.2}
The support of a permuted double Grothendieck polynomial $G_{\omega}^{\gamma}$ is $\{v\mid v\leq_\gamma \omega\}$
\end{theorem}

\begin{proof} We need to show $\mbox{Supp}(G_{\omega})=(X_{\omega})^T$ first. We do it by induction on the length of $v\in W$, $l(v)$, which stands for the minimum number of transpositions in all the possible choices of word expressions of $v$.

For $\omega=id$, $G_{id}$ is just the top Grothendieck polynomial. It is non-zero only at the identity and zero at all the other elements. Assume the inductive hypothesis that $\mbox{Supp}(G_{\omega})=(X_{\omega})^T$ for all $l(\omega)\leq l-1$. Consider $v\in W,l(v)=l$, write $v=s_{i_1}s_{i_2}...s_{i_l}$ where each $s_{i_j}$ is a transposition of elements $i_j,i_j +1$, let $\omega=vs_{{i_l}}=s_{i_1}...s_{i_{l-1}}$, so $l(\omega)=l-1$ and 
\begin{eqnarray}
G_v|_z &=& \pi_{v^{-1}}G|_z=\pi_{i_{l}}\pi_{i_{l-1}}...\pi_{i_1}G|_z=\pi_{i_l}G_{\omega}|_z\nonumber\\
&=&\frac{x_{i_l}G_{\omega}(x,y)-x_{i_l+1}G_{\omega}(x_{s_{i_l}},y)}{x_{i_l}-x_{i_l+1}}|_z\nonumber\\
&=&\frac{y_{z(i_l)}G_{\omega}(y_z,y)-y_{z(i_l+1)}G_{\omega}(y_{zs_{i_l}},y)}{y_{z(i_l)}-y_{z(i_l+1)}} \label{equation3}
\end{eqnarray} 
First, to prove that $\mbox{Supp}(G_v)\subset(X_v)^T$, suppose that $z\not\in(X_v)^T$, then $z\not\in(X_{\omega})^T$ since $\omega\leq v$. Since $l(\omega)=l-1$, we have $z\not\in\mbox{Supp}(G_{\omega})$. That is $G_{\omega}(y_z,y)=0$. Hence, \[G_v|_z=\frac{-y_{z(i_{l}+1)}G_{\omega}(y_{zs_{i_l}},y)}{y_{z(i_l)}-y_{z(i_l+1)}}\]We claim that it is zero. If it were not zero, then $G_{\omega}(y_{zs_{i_l}},y)=G_{\omega}(x,y)|_{zs_{i_l}}\neq 0$. Equivalently, $zs_{i_l}\in \mbox{Supp}(G_{\omega})=(X_{\omega})^T$. If $z<zs_{i_l}$, then $z\in(X_{\omega})^T$ which contradicts $z\not\in\mbox{Supp}(G_{\omega})$ shown before. If $z>zs_{i_l}$, then $s_{i_l}$ increases the length of $zs_{i_l}$. Then $zs_{i_l}\in(X_{\omega})^T$ implies that $z\in(X_{v})^T$ which contradicts $z\not\in(X_{v})^T$. So the claim is proved. i.e. $z\notin (X_v)^T\Rightarrow G_{v}|_z=0 \Leftrightarrow z\not\in\mbox{Supp}(G_{v})$.

Second, we need to prove that $(X_{v})^T\subset\mbox{Supp}(G_{v})$. Suppose that $z\not\in\mbox{Supp}(G_{v})$, i.e. $G_{v}|_z=0$. Assume that $z\in(X_{v})^T$. From (\ref{equation3}), 
\begin{equation} \label{equation4}
y_{z(i_l)}G_{\omega}(y_{z},y)=y_{z(i_{l}+1)}G_{\omega}(y_{zs_{i_l}},y)
\end{equation} Now there are two cases, $z=v$ and $z\neq v$. We consider these two cases separately.

If $z=v$, then $z\not\le w$ (since $l(\omega)=l-1$ and $l(z)=l(v)=l$)$\Leftrightarrow z\not\in(X_{\omega})^T=\mbox{Supp}(G_{\omega})\Leftrightarrow G_{\omega}|_z=0\Leftrightarrow G_{\omega}(y_z,y)=0\Leftrightarrow G_{\omega}(y_{zs_{i_l}},y)=0$. The last equality is by (\ref{equation4}). So we now have $G_{\omega}(x,y)|_{zs_{i_l}}=0\Leftrightarrow zs_{i_l}\not\in\mbox{Supp}(G_{\omega})=(X_{\omega})^T$. Since $zs_{i_l}=vs_{i_l}=\omega\in(X_{\omega})^T$, it's a contradiction.

If $z\neq v$, then $l(z)<l(v)$, then $l(z)\leq l-1$. Let $t\in W$ with $l(t)=l-1$ such that $z\leq t$. Although $t$ may not be the same as $\omega$ but $t=v's_{i_j}$ for some $j\in{1,...,l}$ ($v'$ is another word expression for $v$) By our inductive hypothesis, $\mbox{Supp}(G_{t})=(X_{t})^T$, so
\begin{equation} \label{equation5}
z\in\mbox{Supp}(G_{t})\Leftrightarrow G_{t}(y_z,y)=G_{t}(x,y)|_z\neq 0
\end{equation} But $zs_{i_j}\not\le t$ implies that $zs_{i_j}\not\in(X_{t})^T=\mbox{Supp}(G_{t})$. By (\ref{equation4}), (but now we have $\omega$ replaced by $t$), $G_{t}(y_{zs_{i_j}},y)=0$. By (\ref{equation3}) and (\ref{equation5}), we have $G_{v}|_z\neq 0$ contradicting our initial assumption that $z\notin\mbox{Supp}(G_v)$.

Hence, we have $z\not\in\mbox{Supp}(G_{v})\Rightarrow z\not\in(X_{v})^T$. The induction step is done.

Then we need to show that the statement holds for the permuted double Grothendieck polynomials, i.e. $\mbox{Supp}(G_{\omega}^{\gamma})=(X_{\omega}^{\gamma})^T$. By definition, $G_{\omega}^{\gamma}(x,y)=G_{\gamma^{-1}\omega}(x,y_{\gamma})$, so, \[\mbox{Supp}G_{\gamma^{-1}\omega}(x,y)=(X_{\gamma^{-1}\omega})^T=\{v\in W\mid v\leq\gamma^{-1}\omega\}\]By permuting the $y$'s variables by $\gamma$, we obtain 
\begin{eqnarray}
\mbox{Supp}(G_{\omega}^{\gamma})&=& \mbox{Supp}G_{\gamma^{-1}\omega}(x,y_{\gamma})\nonumber\\
&=&\{\gamma v\in W\mid v\leq\gamma^{-1}\omega\}\nonumber\\
&=&\{v\in W\mid \gamma^{-1}v\leq \gamma^{-1}\omega\}\nonumber\\
&=&\{(X_{\omega}^{\gamma})^T\}\nonumber
\end{eqnarray}
\end{proof}

\subsection{Main theroem} \label{section5.6}

In this subsection, we prove the following result:
\begin{theorem} \label{theorem5.3}
Let $\mathcal{O}_{\lambda}$ be a generic coadjoint orbit of $SU(n)$. Then \[K^{\ast}(\mathcal{O}_{\lambda}//T(\mu))\cong \frac{\mathbb{Z}[x_1,...,x_n,y_1^{\pm 1}]}{(I,((\prod\nolimits^n_{i=1}y_i)-1),\pi_v G(x,y_r))}\]for all $v,r\in S_n$ such that $\sum\nolimits^n_{i=k+1}\lambda_{v(i)}<\sum\nolimits^n_{i=k+1}\mu_{r(i)}$ for some $k=1,...,n-1$. $I$ is the difference between $e_i(x_1,...,x_n)-e_i(y_1,...,y_n)$ for all $i=1,...,n$, where $e_i$ is the $i$-th elementary symmetric polynomial.
\end{theorem}

It is a $K$-theoretic analogue of the main result in \cite{GGoldin}.

To make the symplectic picture more explicit, we denote $M=\mathcal{O}_{\lambda}\approx SU(n)/T$ to be the generic coadjoint orbit. So we have $K_T^{\ast}(M)=K_T^{\ast}(\mathcal{O}_{\lambda})=K_T^{\ast}(Fl(\mathbb{C}^n))$. For $\lambda\in\mathfrak{t}^{\ast}, \lambda=(\lambda_1,...,\lambda_n)$, assume that  $\lambda_1>\lambda_2>...>\lambda_n$, and $\lambda_1+...+\lambda_n=0$. Since $M=\mathcal{O}_{\lambda}$ is compact, $M^T$ has only a finite number of points. The kernel of the Kirwan map $\kappa$ is generated by a finite number of components, see Theorem \ref{corollary3.1} and \cite{HHaradaLandweber}. More specifically, let $M_{\xi}^{\mu}\subset M, \xi\in\mathfrak{t}$ be the set of points where the image under the moment map $\phi$ lies to one side of the hyperplane $\xi^{\perp}$ through $\mu=(\mu_1,...,\mu_n)\in\mathfrak{t}^{\ast}$, i.e. \[M_{\xi}^{\mu}=\{m\in M\mid\langle\xi,\phi(m)\rangle\leq\langle\xi,\mu\rangle\}\]Then the kernel of $\kappa$ is generated by \[K_{\xi}=\{\alpha\in K_T^{\ast}(M)\mid\mbox{Supp}(\alpha)\subset M_{\xi}^{\mu}\}\]That is, \[\ker(\kappa)=\sum_{\xi\in\mathfrak{t}}K_\xi\]

Now, we are going to compute the kernel explicitly. Our proof is similar to the results in \cite{GGoldin}. In \cite{GGoldin}, Goldin proved a very similar result in rational cohomology by using the permuted double Schubert polynomials as a linear basis of $H_T^{\ast}(M)$ over $H_T^{\ast}(pt)$. In $K$-theory, the permuted double Grothendieck polynomials are used as a linear basis of $K_T^{\ast}(M)$ over $K_{T}^{\ast}(pt)\cong R(T)$. The following lemma will be used in our proof of Theorem \ref{theorem5.3}:

\begin{lemma} \label{lemma5.4}
Let $\mathcal{O}_\lambda$ be a generic coadjoint orbit of $SU(n)$ through $\lambda\in\mathfrak{t}^\ast$. Let $\alpha\in K_T^\ast(\mathcal{O}_\lambda)$ be a class with $\mbox{Supp}(\alpha)\subset(\mathcal{O_\lambda})_\xi^\mu$. Then there exists some $\gamma\in W$ such that if $\alpha$ is decomposed in the $R(T)$-basis $\{G_\omega^\gamma\}_{\omega\in W}$ as \[\alpha=\sum_{\omega\in W}a_\omega^\gamma G_\omega^\gamma\]where $a_\omega^\gamma\in R(T)$, then $a_\omega^\gamma\neq 0$ implies $\mbox{Supp}(G_\omega^\gamma)\subset(\mathcal{O_\lambda})_\xi^\mu$. Indeed, $\gamma$ can be chosen such that $\xi$ attains its minimum at $\phi(\lambda_\gamma)$, where $\lambda_\gamma=(\lambda_{\gamma^{-1}(1)},...,\lambda_{\gamma^{-1}(n)})\in\mathfrak{t}^\ast$.
\end{lemma}

\begin{proof} The proof is essentially the same as Theorem 3.1 in \cite{GGoldin}.
\end{proof}

\begin{proof} [Proof of Theorem \ref{theorem5.3}]: Let $e_i$ be the coordinate functions on $\mathfrak{t}^\ast$. That is, for $\lambda=(\lambda_1,\lambda_2,...,\lambda_n)\in\mathfrak{t}^\ast$, $e_i(\lambda)=\lambda_i$. For $\gamma\in S_n$, define $\eta_k^\gamma$by \[\eta_k^\gamma=\sum_{i=k+1}^n e_{\gamma(i)}\] We compute $M_{\eta_k^\gamma}^{\mu}$ explicitly: \begin{eqnarray}
M_{\eta_k^\gamma}^{\mu}&=&\{m\in M\mid\langle\eta_k^\gamma , \phi(m)\rangle\leq\langle\eta_k^\gamma,\mu\rangle\}\nonumber\\
&=&\{m\in M\mid\eta_k^{\gamma}(\phi(m))\leq\eta_k^{\gamma}(\mu)\}\nonumber\\
&=&\{m\in M\mid\eta_k^{\gamma}(\phi(m))\leq\sum\nolimits_{i=k+1}^n \mu_{\gamma(i)}\}\nonumber
\end{eqnarray}For any $\omega\in W$,
\begin{eqnarray}
\eta_k^{\gamma}(\lambda_{\omega})&=&\sum_{i=k+1}^n e_{\gamma(i)}(\lambda_{\omega})=\sum_{i=k+1}^n e_{\gamma(i)}(\lambda_{\omega^{-1}(1)},...,\lambda_{\omega^{-1}(n)})\nonumber\\
&=& \sum_{i=k+1}^n \lambda_{\omega^{-1}\gamma(i)}
\nonumber
\end{eqnarray} Notice that $\eta_k^{\gamma}$ attains minimum at $\lambda_{\gamma}$ (due to our assumption that $\lambda_1\geq\lambda_2\geq...\geq\lambda_n$) and respects the permuted Bruhat ordering, i.e. \[\eta_k^{\gamma}(\lambda_{v})\leq\eta_k^{\gamma}(\lambda_{\omega})\]if $v\leq_\gamma \omega$. By restriction to the domain $\mbox{Supp}(G_{\omega}^{\gamma})=(X_{\omega}^{\gamma})^T=\{v\in W\mid v\leq_\gamma w\}=\{v\in W\mid \gamma^{-1}v\leq\gamma^{-1}\omega\}$, $\eta_k^\gamma$ attains its maximum at $\lambda_{\omega}$ and minimum at $\lambda_{\gamma}$. If $\eta_k^{\gamma}(\lambda_{\omega})=\sum\nolimits_{i=k+1}^n \lambda_{\omega^{-1}\gamma(i)}<\sum\nolimits_{i=k+1}^n \mu_{\gamma(i)}$, then  for $v\in(X_{\omega}^{\gamma})^T$, \[ \eta_k^{\gamma}(\lambda_{v})=\sum_{i=k+1}^n \lambda_{v^{-1}\gamma(i)}\leq\sum_{i=k+1}^n\lambda_{\omega^{-1}\gamma(i)}<\sum_{i=k+1}^n \mu_{\gamma(i)}\]and hence \[\mbox{Supp}(G_{\omega}^{\gamma})=(X_{\omega}^{\gamma}
)^T=\{v\in W\mid\gamma^{-1}v\leq\gamma^{-1}\omega\}\subset M_{\eta_k^{\gamma}}^{\mu}\]
Since $G_{\omega}^{\gamma}(x,y)=\pi_{\omega^{-1}\gamma}G(x,y_{\gamma})$, we have $\pi_vG(x,y_{\gamma})\in\ker(\kappa)$ if $\sum\nolimits_{i=k+1}^n \lambda_{v(i)}<\sum\nolimits_{i=k+1}^n \mu_{\gamma(i)}$.

For the other direction, we need to show that the classes $\pi_v G(x,y_\gamma)$ with $v,\gamma\in W$ having the property that $\sum\nolimits_{i=k+1}^n \lambda_{v(i)}<\sum\nolimits_{i=k+1}^n \mu_{\gamma(i)}$ for some $k\in\{1,...,n-1\}$ actually generate the whole kernel. Let $\alpha\in K_T^{\ast}(M)$ be a class in $\ker(\kappa)$, so $\mbox{Supp}(\alpha)\subset M_{\xi}^{\mu}$ for some $\xi\in\mathfrak{t}$. We take $\gamma\in W$ such that $\xi(\lambda_{\gamma})$ attains its minimum. Decompose $\alpha$ over the $R(T)$-basis $\{G_\omega^\gamma\}_{\omega\in W}$,
\[\alpha=\sum_{\omega\in W} a_{\omega}^{\gamma}G_{\omega}^{\gamma}\]where $a_{\omega}^{\gamma}\in R(T)$. By Lemma \ref{lemma5.4}, we need to show that $\mbox{Supp}(G_{\omega}^{\gamma})\subset M_{\eta_k^{\gamma}}^{\mu}$ for some $k$. Since $\eta_k^\gamma$ is preserved by the permuted Bruhat ordering and attains its maximum at $\lambda_{\omega}$ in the domain $\mbox{Supp}(G_{\omega}^{\gamma})$, we just need to show that \begin{equation} \label{equation6}
\eta_{k}^{\gamma}(\lambda_{\omega})<\eta_k^{\gamma}(\mu)
\end{equation} for some $k$. It is actually purely computational: Suppose (\ref{equation6}) does not hold for all $k$. We have
\begin{eqnarray}
\lambda_{\omega^{-1}\gamma(n)} &\geq& \mu_{\gamma(n)}\nonumber\\
&\vdots& \nonumber\\
\lambda_{\omega^{-1}\gamma(2)}+...+\lambda_{\omega^{-1}\gamma(n)}&\geq& \mu_{\gamma(2)}+...+\mu_{
\gamma(n)}\nonumber
\end{eqnarray} For $\xi=\sum\nolimits_{i=1}^n b_ie_i, b_1,...,b_n\in\mathbb{R}$ (recall that $\xi$ attains its minmum at $\lambda_\gamma$ by our choice of $\gamma\in W$), we have $\xi(\lambda_\gamma)\leq\xi(\lambda_{s_i\gamma})$ where $s_i$ is a transposition of $i$ and $i+1$. And hence
\[b_i \lambda_{\gamma^{-1}(i)}+b_{i+1}\lambda_{\gamma^{-1}(i+1)}\leq b_i \lambda_{\gamma^{-1}(i+1)}+b_{i+1}\lambda_{
\gamma^{-1}(i)}\]By our assumption that $\lambda_i>\lambda_{i+1}$, we get $b_{\gamma(i)}\leq b_{\gamma(i+1)}$. And hence $b_{\gamma(1)}\leq b_{\gamma(2)}\leq...\leq b_{\gamma(n)}$. Then,
\begin{eqnarray}
(b_{\gamma(n)}-b_{\gamma(n-1)})\lambda_{\omega^{-1}\gamma(n)}&\geq& (b_{\gamma(n)}-b_{\gamma(n-1)})\mu_{\gamma(n)}\nonumber\\
(b_{\gamma(n-1)}-b_{\gamma(n-2)})(\lambda_{\omega^{-1}\gamma(n-1)}+\lambda_{\omega^{-1}\gamma(n)})&\geq& (b_{\gamma(n-1)}-b_{\gamma(n-2)})(\mu_{\gamma(n-1)}+\mu_{\gamma(n)})\nonumber\\
&\vdots& \nonumber\\
(b_{\gamma(2)}-b_{\gamma(1)})(\lambda_{\omega^{-1}\gamma(2)}+...+\lambda_{\omega^{-1}\gamma(n)})&\geq& (b_{\gamma(2)}-b_{\gamma(1)})(\mu_{\gamma(2)}+...+\mu_{\gamma(n)})\nonumber\\ \nonumber
\end{eqnarray} Using $\sum\nolimits_{i=1}^n\lambda_i=0=\sum\nolimits_{i=1}^n\mu_i$ and summing up all the above inequalities to get 
\begin{eqnarray}
\sum_{i=1}^n b_{\gamma(i)}\lambda_{\omega^{-1}\gamma(i)}&\geq&\sum_{i=1}^n b_i\mu_i\nonumber\\
\Leftrightarrow\sum_{i=1}^n b_i\lambda_{\omega^{-1}(i)}&\geq&\sum_{i=1}^n b_i\mu_i\nonumber\\
\Leftrightarrow \xi(\lambda_{\omega})&\geq&\xi(\mu)\nonumber
\end{eqnarray} the last inequality contradicts $\mbox{Supp}(\alpha)\subset M_{\xi}^{\mu}$ since $\lambda_{\omega}$ has the property that $\omega\in\mbox{Supp}(\alpha)$. So (\ref{equation6}) is true.

So the kernel $\ker(\kappa)$ is generated by the set $\pi_v G(x,y_\gamma)$ for $v,\gamma\in W$ satisfying $\sum_{i=k+1}^n\lambda_{v(i)}<\sum\nolimits_{i=k+1}^n \mu_{\gamma(i)}$ for some $k=1,...,n-1$. By (\ref{equation2}) and the surjectivity of the Kirwan map $\kappa$, \[\kappa\colon K_T^\ast(SU(n)/T)=K_T^{\ast}(\mathcal{O}_\lambda)\to K_T^\ast(\phi^{-1}(\mu))\cong K^\ast(\mathcal{O}_\lambda//T(\mu))\]It implies that \[K^{\ast}(\mathcal{O}_\lambda//T(\mu))=K_T^{\ast}(\mathcal{O}_\lambda)/\ker(\kappa)\] With $\ker(\kappa)$ explicitly computed and by (\ref{equation2}), Theorem \ref{theorem5.3} is proved.
\end{proof}

\section*{Acknowledgements}
The author would like to thank Professor Sjamaar for all his encouragements, advice, teaching and pointing out the calculation mistakes in the first edition of this paper. The author is also grateful to the referee for the critical comments.

\end{document}